\documentclass{amsart}
\usepackage{amsfonts,amssymb,amscd,amsmath,enumerate,verbatim,calc,amsthm,mathrsfs,amsxtra}
\usepackage[all,2cell,ps]{xy}
\usepackage{enumerate,verbatim}
\usepackage[pagebackref]{hyperref} %(for citesion colour)

\newcommand{\fg}{finitely generated }

\newcommand{\ff}{\text{if and only if}}
\newcommand{\wrt}{with respect to}

\newcommand{\m}{\mathfrak{m} }

\newcommand{\N}{\mathcal{N} }

\newcommand{\p}{\mathfrak{p} }

\newcommand{\R}{\mathcal{R} }

\newcommand{\Z}{\mathbb{Z} }

\newcommand{\grade}{\operatorname{grade}}
\newcommand{\depth}{\operatorname{depth}}

\newcommand{\projdim}{\operatorname{projdim}}

\newcommand{\Ass}{\operatorname{Ass}}
\newcommand{\Hom}{\operatorname{Hom}}
\newcommand{\Ext}{\operatorname{Ext}}
\newcommand{\Tor}{\operatorname{Tor}}

\theoremstyle{plain}

\newtheorem{theorem}{Theorem}[section]
\newtheorem{corollary}[theorem]{Corollary}
\newtheorem{lemma}[theorem]{Lemma}

\newtheorem{some new invariant}[theorem]{Some new invariant}
\theoremstyle{definition}

\newtheorem{s}[theorem]{}
\newtheorem{remark}[theorem]{Remark}

\newtheorem{note}[theorem]{Note}
\newtheorem{setup}[theorem]{Set-up}
\theoremstyle{remark}

\numberwithin{equation}{theorem}

\begin{document}

\title[Asymptotic depth]{Asymptotic depth of Ext modules over complete intersection rings}

\author{Provanjan Mallick}%[P. Mallick]
\address{Department of Mathematics, Indian Institute of Technology Bombay, Powai, Mumbai 400076, India}
\email{prov786@math.iitb.ac.in}

\author{Tony J. Puthenpurakal}%[T. J. Puthenpurakal]
\address{Department of Mathematics, Indian Institute of Technology Bombay, Powai, Mumbai 400076, India}
\email{tputhen@math.iitb.ac.in}

\date{\today}

%\subjclass[2010]{Primary 13A17, 13A30, 13D07; Secondary 13A15, 13H10}%
%\keywords{Asymptotic associate primes; Asymptotic grade; Associated graded rings and modules; Local cohomology; Tor; Complete intersections}

\begin{abstract}
	Let $(A,\m)$ be a local complete intersection ring and let $I$ be an ideal in $A$. Let $M, N$ be finitely generated
	$A$-modules. Then for $l = 0,1$, the values
$\depth \Ext^{2i+l}_A(M, N/I^nN)$ become independent of $i, n$ for $i,n \gg 0$.
We also show that if $\p$ is a prime ideal in $A$ then the $j^{th}$ Bass numbers
$\mu_j\big(\p,\Ext^{2i+l}_A(M,N/{I^nN})\big)$ has polynomial growth in $(n,i)$ with rational coefficients for all sufficiently large $(n,i)$.

\end{abstract}
\maketitle
\section{Introduction}
Let $A$ be a Noetherian ring and let $I$ be an ideal in $A$. Let $M$ be a finitely generated $A$-module. By a result of
Brodmann \cite{B79-1} the set $\Ass_A (M/I^nM)$ and the set
$\Ass_A( I^nM/I^{n+1}M)$ are eventually stable. Later \cite{B79-2} he showed that if $J$ is an ideal in $A$ then
$\grade(J, M/I^nM)$ and $\grade(J, I^nM/I^{n+1}M)$ are eventually constant.  Recently in a far-reaching generalization 
T. Se \cite{Se17}   showed that if $F$ is a \emph{coherent} covariant $A$-linear functor  from the category of finitely  $A$-modules to itself the sets $F(M/I^nM), F(I^nM/I^{n+1}M)$  and the values 
$\grade(J, F(M/I^nM))$ and $\grade(J, F(I^nM/I^{n+1}M))$ become independent of $n$ for large $n$. 
We note that if $N$ is a finitely generated $A$-module then the functors $\Tor^A_i(N,-)$ and $\Ext^i_A(N,-)$ are
coherent for all $i \geq 0$, cf. \cite{Se17}.

For the rest of the introduction let us assume that $(A,\m)$ is a local complete intersection.
In \cite{Put13} the second author proved that if  $\N = \bigoplus_{i \geq 0} \N_n$ is any finitely generated module over $A[It]$, the Rees-algebra of $I$, then for $l = 0,1$  the sets
$\Ass_A \big(\Ext^{2i+l}_A(M, \N_n)\big)$ are stable for large $i,n$. We note two significant cases are when $\N = \bigoplus_{n \geq 0}I^nN$ and when $\N = \bigoplus_{n \geq 0}I^nN/I^{n+1}N$ for a finitely generated $A$-module $N$.
Later the second author (with D. Ghosh) \cite{GP} proved that if $N$ is a finitely generated 
$A$-module then for $l = 0,1$  the sets
$\Ass_A\big( \Ext^{2i+l}_A(M, N/I^nN)\big)$ are stable for large $i,n$.
In view of earlier results a natural question is that if $J$ is an ideal in $A$ then, for $l = 0,1$, are the values
\begin{enumerate}
\item
$\grade\big(J, \Ext^{2i+l}_A(M, \N_n)\big)$ where $\N = \bigoplus_{n \geq 0}\N_n$ is a finitely generated module over $A[It]$
\item
$\grade\big(J, \Ext^{2i+l}_A(M, N/I^nN)\big)$ where $N$ is a finitely generated $A$-module
\end{enumerate}
 independent of $i,n$ for $i,n \gg 0$. The first question has an affirmative answer (see Corollary 3.3). For the second question we prove
\begin{theorem}\label{main-tjp-prov}
Let $(A,\m)$ be a local complete intersection ring and let $I$ be an ideal in $A$. Let $M, N$ be finitely generated $A$-modules. Then for $l = 0,1$ the values
$\depth\big( \Ext^{2i+l}_A(M, N/I^nN)\big)$ become independent of $i, n$ for $i,n \gg 0$.
\end{theorem}
Our techniques do not work for arbitrary ideals of $A$.

Finally  we may enquire about growth of Bass numbers of $\Ext^{2i+l}_A(M, N/I^nN)$
as $i,n$ vary. If $\p$ is a prime ideal in $A$ and $E$ is an $A$-module then let $\mu_j(\p, E)$ denote the $j^{th}$ Bass number 
of $E$ \wrt \ $\p$.
We prove
\begin{theorem}\label{growth-bass}
Let $(A,\m)$ be a local complete intersection ring and let $I$ be an ideal in $A$.
For any prime ideal $\p$ of $A$ and for every fixed $l=0,1$ we have 
$\mu_j\big(\p,\Ext^{2i+l}_A(M,N/{I^nN})\big)$ has polynomial growth in $(n,i)$ with rational coefficients for all sufficiently large $(n,i)$.
\end{theorem}
\begin{remark}
We note that Theorem \ref{main-tjp-prov} does not follow from Theorem \ref{growth-bass} as $\bigoplus_{i,n}\Ext^{i}_A(M, N/I^nN)$ is not finitely generated over some bigraded ring.
\end{remark}

We now describe in brief the contents of this article. In section two we show that the modules under consideration have 
a natural structure of a bigraded module over an appropriate bigraded ring. In section three we prove some preliminary results.
In section four we prove our main result Theorem \ref{main-tjp-prov}. Finally in section five we prove Theorem \ref{growth-bass}.
\section{Module Structure}
In this section we give the module structures which we are going to use in order to prove our main results. 

Let $Q$ be a commutative Noetherian ring and $\boldsymbol{f} = f_1, \ldots, f_c $ a $Q$-regular sequence. Set $A=Q/{(\boldsymbol{f})}$. Let $M$ and $D$ be two $A$-modules, where $M$ is a finitely generated $A$-module. We assume $\projdim_Q(M)< \infty$. We will not change $M$ throughout our discussion.
\s \label{m1}

Let  $\mathbb{F}$ : $\cdots\longrightarrow F_n\longrightarrow \cdots\longrightarrow F_1\longrightarrow F_0\longrightarrow 0$ be a projective resolution of $M$ by finitely generated free $A$-modules. Let $t_j : \mathbb{F}(+2)\longrightarrow \mathbb{F},~~ 1\leq j\leq c$ be the Eisenbud operators defined by $\boldsymbol{f} = f_1, \ldots, f_c $ (see[\cite{Eis80}, section 1]). By [\cite{Eis80}, 1.4] the maps $t_j$ are determined unique upto homotopy . In particular, they induce well defined maps
\[t_j : \Ext^i_A(M,D)\longrightarrow \Ext^{i+2}_A(M,D)  ~ \text{for all }i ~\text{and} ~1\leq j\leq c.\]
The maps $t_j (j=1,\ldots, c)$ commutes upto homotopy. Thus 
\[\Ext^*_A(M,D)=\bigoplus _{i\geq 0}\Ext^i_A(M,D)\] 
turns into a graded $\mathbf{T}:=A[t_1,\ldots, t_c]$-module, where $\mathbf{T}$ is a graded polynomial ring over $A$ in the homology operators $t_j$ defined by $\boldsymbol{f}$ with degree of each $t_j$ is $2$. Furthermore the structure depends only on $\boldsymbol{f}$, are natural in both module arguments and commute with the connecting maps induced by short exact sequences.

\begin{note}\label{no1}
~$ \big(\Ext^*_A(M,D)\big)_{even}=\bigoplus _{i\geq 0}\Ext^{2i}_A(M,D)$ and $ \big(\Ext^*_A(M,D)\big)_{odd}=\\ \bigoplus _{i\geq 0}\Ext^{2i+1}_A(M,D)$ are graded $\mathbf{T}:=A[t_1,\ldots, t_c]$-submodule of $\Ext^*_A(M,D)$. So when we are considering $ \big(\Ext^*_A(M,D)\big)_{even}$ or  $ \big(\Ext^*_A(M,D)\big)_{odd}$ as $\mathbf{T}$-module we can think it as $\mathbf{T}^*$-module, where $\mathbf{T}^*$ is same as $\mathbf{T}$  but the degree of each $t_j$ in $\mathbf{T}^*$ is $1$. Furthermore as $\mathbf{T}^*$-module deg$\big( \Ext^{2i+t}_A(M,D)\big)=i$ for $t=0,1$. 
\end{note}

\s\label{gg}
If $\projdim_Q(M)$ finite and $D$ is finitely generated $A$-module, then Gulliksen [\cite{G},3.1] proved that $\Ext^*_A(M,D)$ is finitely generated $\mathbf{T}:=A[t_1,\ldots, t_c]$-module.

\begin{note}\label{no2}
 In that case we will have $ \big(\Ext^i_A(M,D)\big)_{even}$ and $ \big(\Ext^i_A(M,D)\big)_{odd}$ are finitely generated $\mathbf{T}^*:=A[t_1,\ldots, t_c]$-modules.
\end{note}

\begin{setup}\label{1}
Let $Q$ be a Noetherian ring of finite Krull dimension, and let $\boldsymbol{f} = f_1, \ldots, f_c $ a $Q$-regular sequence. Set $A=Q/{(\boldsymbol{f})}$.
\end{setup}

\s 
 Along with Set-up~\ref{1}, let $\mathbb{F}$ be a projective resolution of $M$ by finitely generated free $A$-modules and let $t_j : \mathbb{F}(+2)\longrightarrow \mathbb{F},~~ 1\leq j\leq c$ be the Eisenbud operators as in \ref{m1}. Let $I$ be an ideal of $A$. Set $\R(I) :=\bigoplus_{n\geq 0} I^nt^n$, the Rees ring of A with respect to I. Let $\mathcal{N} = \bigoplus_{n\geq 0}N_n$ be a $\R(I)$-module. Let $a\in I^s$. Consider $u = at^s\in\R(I)_s$. The map 
 \[N_n\xrightarrow{u} N_{n+s}\]
yields a commutative diagram
\[
\xymatrix
{
\Hom(\mathbb{F}, N_n) \ar@{->}[d]^{u}\ar@{->}[r]_{t_j}
&\Hom(\mathbb{F}, N_n)(+2) \ar@{->}[d]^{u}
\\
\Hom(\mathbb{F}, N_{n+s}) \ar@{->}[r]_{t_j}
&\Hom(\mathbb{F}, N_{n+s})(+2).
}
\]

Taking homology gives that $\mathcal{E}(\mathcal{N}): =\bigoplus_{i\geq 0}\bigoplus_{n\geq 0}\Ext^i_A(M,N_n)$ is a bigraded $\mathcal{T}=\R(I)[t_1,\ldots, t_c]$-module, where degree of $t_j$ is $(0,2)$ for $1\leq j\leq c$ and degree of $ut^n$ is $(n,0)$ for $u\in I^n$ and $n\in \Z$.
Now let us recall the following result from [\cite{Put13}, 1.1].

\begin{theorem}\label{t1}
	Along with Set-up~\ref{1}, suppose $M$ be a finitely generated $A$-module with $\projdim_Q(M)<\infty$. Let $I$ be an ideal of $A$, and let $\mathcal{N} = \bigoplus_{n\geq 0}N_n$ be a finitely generated $\R(I)$-module. Then 
	\[\mathcal{E}(\mathcal{N}): =\bigoplus_{i\geq 0}\bigoplus_{n\geq 0}\Ext^i_A(M,N_n)\] 
	is a finitely generated bigraded $\mathcal{T}=\R(I)[t_1,\ldots, t_c]$-module.
\end{theorem}

\begin{note}\label{no3}
$ \big(\mathcal{E}(\mathcal{N})\big)_{even}=\bigoplus_{i\geq 0}\bigoplus_{n\geq 0}\Ext^{2i}_A(M,N_n)$ and~ $ \big(\mathcal{E}(\mathcal{N})\big)_{odd}=\\ \bigoplus_{i\geq 0}\bigoplus_{n\geq 0}\Ext^{2i+1}_A(M,N_n)$ are finitely generated bigraded $\mathcal{T}^*=\R(I)[t_1,\ldots, t_c]$-module, where $\mathcal{T}^*$ is same as $\mathcal{T}$  but the degree of each $t_j$ in $\mathcal{T}^*$ is $(0,1)$.
\end{note}

 \section{Some Preliminary results}
 \s 
 Throughout this section we are assuming $\mathbf{T}, \mathbf{T}^*$ as in  Note~\ref{no1}~ and $\mathcal{T}, \mathcal{T}^*$ as in  Note~\ref{no3}. Let $I$ be an ideal of $A$. Set $\R(I) :=\bigoplus_{n\geq 0} I^nt^n$, the Rees ring of A with respect to I.

 \begin{theorem}\label{last1}
 Along with Set-up~\ref{1}, suppose $M$ be a finitely generated $A$-module with $\projdim_Q(M)<\infty$.	Let $\mathcal{N} = \bigoplus_{n\geq 0}N_n$ be a finitely generated $\R(I)$-module. Let $J$ be an ideal of $A$. Then we have $\grade\big(J,\Ext^{2i+t}_A(M,N_n)\big)$ is constant for every fixed $t=0,1$ and for all $n,i\gg 0$.
 \end{theorem}
 \begin{proof}
 	We prove the Theorem for $t=0$ only. For $t=1$ the proof is similer.
 	
 	It is well known that for fixed $n,i$ 
 	\begin{equation*}
 	\grade\big(J,\Ext^{2i}_A(M,N_n)\big) = \min\{ l : \Ext^l_A\big( A/J,\Ext^{2i}_A(M,N_n)\big)  \neq 0 \}.
 	\end{equation*} 
 	Set $\mathcal{E}=\bigoplus_{i\geq 0}\bigoplus_{n\geq 0}\Ext^i_A(M,N_n)$ and fix $l$. By virtue of~\ref{t1} $\mathcal{E}$ is finitely generated bigraded $\mathcal{T}=\R(I)[t_1,\ldots, t_c]$-module. So for fixed $l$ we have $\Ext^l_A(A/J,\mathcal{E})$ is finitely generated bigraded $\mathcal{T}$-module.
 	Hence by [\cite{West}, Proposition 5.1], there exist $(n_l,i_l)\in \mathbb{N}^2_0$ such that %for each fixed $t=0,1,$~
 	 $$\Ass_A\big(\Ext^l_A(A/J,\mathcal{E}_{(n,2i)})\big)=\Ass_A\big(\Ext^l_A(A/J,\mathcal{E}_{(n_l,2i_l)})\big)~\text{for all}~ (n,i)\geq (n_l,i_l),$$ 
 	where $\mathbb{N}_0$ denotes the set of non-negative integers.  
 	For each fixed $(n,i)\in \mathbb{N}^2_0$ and for each $u\geq 0$  set
 	$$D^{u}_{(n,2i)}:=\Ext^u_A\big( A/J,\mathcal{E}_{(n,2i)}) \big).$$

 	Hence we have for each fixed $u$ there exist $(n_u,i_u)\in \mathbb{N}^2_0$ such that
 	\begin{equation}\label{aaz}
 	\text{either}~D^{u}_{(n,2i)}=0~\text{ for all }~(n,i)\geq (n_u,i_u)~\text{or }~ D^{u}_{(n,2i)}\neq0~\text{for all}~(n,i)\geq (n_u,i_u).
 	\end{equation}

 	Let $(\alpha,\beta)$ be the maximimum of such $(n_u,i_u)$ for $0\leq u \leq \dim(A)$. Let $ \gamma= \min\{l|D^{l}_{(\alpha,\beta)}\neq0\}$. Hence we have $\grade\big(J,\Ext^{2i}_A(M,N_n)\big)=\gamma$ for all  $(n,i)\geq (\alpha,\beta)$.	
 \end{proof}

\begin{corollary}
 Let $(A,\m)$ be a local complete intersection ring. Let $M$ be a finitely generated $A$-module, let $J$ be an ideal of $A$. Let $\mathcal{N} = \bigoplus_{n\geq 0}N_n$ be a finitely generated graded $\R(I)$-module. Then for fixed $t=0,1$ we have $\grade\big(J,\Ext^{2i+t}_A(M,\mathcal{N}_n)\big)$ is constant for all $n,i\gg 0$.
\end{corollary}
\begin{proof}
	It is well known that for a \fg $A$-module $E$ and for a ideal $J$ of $A$
	\begin{equation}\label{i13}
	\grade(J,E)=\grade(\hat{J},\hat{E}),
	\end{equation}
	where $\hat{J}$ and $\hat{E}$ are completion of $J$ and $E$ respectively. Since $\hat{A}$ is flat over $A$ we have for each fixed $n,i\geq0$ and for each fixed $t=0,1$ $$\widehat{\Ext^{2i+t}_A(M,\mathcal{N}_n)}=\Ext^{2i+t}_{\hat{A}}(\hat{M},\hat{\mathcal{N}_n}).$$

	So we may assume $A$ is complete. So there exist a regular local ring $Q$ with $A=Q/{(f)}$ and $f=f_1,\ldots,f_{c}$ is a $Q$-regular sequence. Also we will have $\projdim_{Q}(M)<\infty$, since $M$ is a \fg $A$-module and so as $Q$-module. Now the Corollary follows from Theorem~\ref{last1}.
\end{proof}

We will need the next result in the following section.
\begin{lemma}\label{ll2}
 Along with Set-up~\ref{1}, further assume $Q$ is a local ring with residue field $k$. Let $M$ be a finitely generated $A$-module with $\projdim_Q(M)<\infty$. Let $S_t=\bigoplus_{i\geq 0}(S_t)_i= \Ext^l_A\big(k,\bigoplus_{i\geq 0}\Ext^{2i+t}_A(M,N/I^nN)\big)$, where $n$ and $l$ are two fixed integers. Then for fixed $t=0,1$ there exist $u\in \mathbb{N}$ such that $\Ass_A\big((S_t)_j\big)=\Ass_A\big((S_t)_u\big)$ for all  $j\geq u$. Furthermore for fixed $t=0,1$ either $\Ass_A((S_t)_j)=\m$ for all $j\geq u$ or empty for all $j\geq u$. Here $\m$ is the unique maximal ideal of $A$.
\end{lemma}
\begin{proof} 
	We prove the Lemma for $t=0$ only. For $t=1$ the proof is similer.

By ~\ref{gg} we have $\bigoplus_{i\geq 0}\Ext^{2i}_A(M,N/I^nN)$ is finitely generated graded $\mathbf{T}^*:=A[t_1,\ldots, t_c]$-module. Hence $S_0=\bigoplus_{i\geq 0}(S_0)_i= \Ext^l_A\big(k,\bigoplus_{i\geq 0}\Ext^{2i}_A(M,N/I^nN)\big)$ is finitely generated graded $\mathbf{T}^*$-module. So each $(S_0)_i$ is  finitely generated $A$-module. Hence there exist $u\in \mathbb{N}$ such that $\Ass_A\big((S_0)_j\big)=\Ass_A\big((S_0)_u\big)$ for all  $j\geq u$. 

Since for each $i\geq0$, $(S_0)_i$ is annhilated by $\m$, we are done.
\end{proof}

\section{Proof of Theorem \ref{main-tjp-prov}}
\s 
We begin by establishing the notation for $\mathbb{N}_0^2$-graded algebras, where $\mathbb{N}_0$ denotes the set of non-negative integers. A ring $R$ is called a $\mathbb{N}_0^2$-graded algebra  if $R=\bigoplus_{i,j\geq 0}R_{(i,j)}$ where each $R_{(i,j)}$ is an additive subgroup
of $R$ such that $R_{(i,j)}\cdot R_{(l,m)}\subseteq R_{(i+l,j+m)}$ for all $(i,j),(l,m)\in \mathbb{N}_0^2$. We say that $R$ is a standard Noetherian $\mathbb{N}_0^2$-graded algebra if $A=R_{(0,0)}$ is Noetherian and $R$ is finitely generated as an $A$ algebra by elements of degree $(1, 0)$ and $(0, 1)$,  i.e., it is generated in total degree one. Let $n=(n_1,n_2)$ and $m=(m_1,m_2)$, where $n,m\in \mathbb{N}_0^2$. We say $n\geq m$ if $n_i\geq m_i$ for all $i=1,2$. We will write $R_{++}$ for the ideal consisting of all sums of homogeneous elements $x_n\in R_n$ such that $n_i\geq1$, for all $i=1,2$. In other words, $R_{++}$ denotes the ideal of $R$ generated by $R_{(1,1)}$. An $R$-module $M$ is called bigraded if $M =\bigoplus_{r,s\in\Z} M_{(r,s)}$, where $M_{(r,s)}$ are additive subgroups of $M$ satisfying $R_{(r,s)}\cdot M_{(l,m)} \subseteq M_{(r+l,s+m)}$ for all $(r, s)\in  \mathbb{N}_0^2$ and $l, m\in \Z$.

Let us recall the following well known Theorem :
\begin{theorem}\label{jk1}
	Let $R=\bigoplus_{i,j\geq 0}R_{(i,j)}$ be a finitely generated standard $ \mathbb{N}_0^2$-graded algebra over a Noetherian local ring $R_{(0,0)}=(A,\m)$. Let $M$ be a finitely generated bigraded $R$-module. Then $H^i_{R_{++}}(M)_{(r,s)} = 0$ for all $r,s\gg 0$ and $i\geq0$.
\end{theorem}
\begin{proof}
See [ \cite{Jkv02}, Theorem 2.3].
\end{proof}

Let us also recall the following result from[\cite{Tk},Theorem 4.4], which we will use to prove our next Theorem.
\begin{theorem}\label{sirkatz}
	Let $R$ denotes a Noetherian standard $\mathbb{N}_0^2$ graded $A$-algebra. Let $M=\bigoplus_{i,j\geq0}M_{(i,j)}$ be a (not necessarily finitely generated) bigraded $R$-module. Set $L:=H^0_{R_{++}}(M)$.  Assume $\Ass_A(M)$ is finite. Consider the following statements:
	\begin{enumerate}
		\item There exist $l\in \mathbb{N}_0^2$ such that $\Ass_A(L_n)=\Ass_A(L_l)$ for all $n\in \mathbb{N}_0^2$ with $n\geq l$.
		\item There exist $h\in \mathbb{N}_0^2$ such that $\Ass_A(M_n)=\Ass_A(M_h)$ for all $n\in \mathbb{N}_0^2$ with $n\geq h$.
	\end{enumerate}
	Then $(1)\implies (2)$.
\end{theorem}

We note that we do not know whether $(1)\implies (2)$ if $\Ass_A(M)$ is not a priori known to be finite.

\begin{theorem}\label{t3}
	Along with Set-up~\ref{1} further assume $Q$ is a local ring with residue field $k$. Let $M$  and $N$ be two finitely generated $A$-module with $\projdim_Q(M) < \infty$, let $I$ be an ideal of $A$. Then for fixed $l$ and for every fixed $t=0,1$ we have that : 
\[\text{either}~\Ext^l_A(k,\Ext^{2i+t}_A(M,N/{I^nN}))=0 ~\text{for all}~ n,i\gg 0\]
\[\text{or}~\Ext^l_A(k,\Ext^{2i+t}_A(M,N/{I^nN}))\neq0 ~\text{for all}~ n,i\gg 0.\]
\end{theorem}
\begin{proof}We prove the Theorem for $t=0$ only. For $t=1$ the proof is similer.

	For each $n\geq0$, consider the exact sequence
	\[0\rightarrow I^nN/I^{n+1}N \rightarrow N/I^{n+1}N\rightarrow N/I^nN\rightarrow 0,\] 
	which induces an exact sequence of $A$-modules(for each $i$):\\
	$$\Ext^{2i}_A(M,I^nN/I^{n+1}N)\longrightarrow \Ext^{2i}_A(M,N/I^{n+1}N)\longrightarrow \Ext^{2i}_A(M,N/I^nN)$$ 
	$$\longrightarrow \Ext^{2i+1}_A(M,I^nN/I^{n+1}N).$$ 
	Taking direct sum over $n,i$ and using the naturality of the Eisenbud operators $t_j$, we have an exact sequence:
	~\[\bigoplus_{n,i\geq 0}\Ext^{2i}_A(M,I^nN/I^{n+1}N)\rightarrow \bigoplus_{n,i\geq 0}\Ext^{2i}_A(M,N/I^{n+1}N))\rightarrow\] \[\bigoplus_{n,i\geq 0}\Ext^{2i}_A(M,N/I^nN) \rightarrow \bigoplus_{n,i\geq 0}\Ext^{2i+1}_A(M,I^nN/I^{n+1}N)\]
	 of bigraded $\mathcal{T}^*$-modules. 
	\[ W=\bigoplus_{n,i\geq 0}W_{n,i}=\bigoplus_{n,i\geq 0}\Ext^{2i}_A(M,I^nN/I^{n+1}N);\] \[V=\bigoplus_{n,i\geq 0}V_{n,i}=\bigoplus_{n,i\geq 0}\Ext^{2i}_A(M,N/I^nN);\]
	\[ \text{Set}~U'=\bigoplus_{n,i\geq 0}U'_{n,i}=\bigoplus_{n,i\geq 0}\Ext^{2i+1}_A(M,I^nN/I^{n+1}N).\]
	So we have an exact sequence $W\rightarrow V(1,0)\rightarrow V\rightarrow U'.$ Now setting $X=\rm Image\big(W\rightarrow V(1,0)\big)$; $Y=\rm Image\big(V(1,0)\rightarrow V\big)$; $Z=\rm Image(V\rightarrow U') $, we obtain the short exact sequences:
	\begin{equation}\label{e12}
	0\longrightarrow X\longrightarrow V(1,0)\longrightarrow Y\longrightarrow 0.
	\end{equation}
	\begin{equation}\label{e13}
	0\longrightarrow Y\longrightarrow V\longrightarrow Z\longrightarrow 0.
	\end{equation}
	Applying $\Hom_A(k,-)$ on \ref{e12} and \ref{e13}, we get the following exact sequences:
	\begin{equation}\label{e14}
	0\longrightarrow C_1 \longrightarrow \Ext^l_A(k,V(1,0))\longrightarrow \Ext^l_A(k,Y)\longrightarrow C_2\longrightarrow 0.
	\end{equation}
	\begin{equation}\label{e15}
	0\longrightarrow C_3 \longrightarrow \Ext^l_A(k,Y) \longrightarrow \Ext^l_A(k,V)\longrightarrow C_4\longrightarrow 0.
	\end{equation}
	We set $C_1=$Image$(\Ext^{l}_A(k,X)\rightarrow \Ext^l_A(k,V(1,0)))$, $C_2=$Image$(\Ext^{l}_A(k,Y)\rightarrow \Ext^{l+1}_A(k,X))$, $C_3=$Image$(\Ext^{l-1}_A(k,Z)\rightarrow \Ext^l_A(k,Y))$, $C_4=$Image$(\Ext^{l}_A(k,V)\rightarrow \Ext^{l}_A(k,Z))$. By virtue of \ref{t1}  ~$U'~\text{and}~W$ are finitely generated bigraded $\mathcal{T}^*$-module ( by  Note~\ref{no3}~), and hence $X$ and $Z$ are so. This implies that $\Ext^l_A(k,X))$, $\Ext^{l+1}_A(k,X))$, $\Ext^{l-1}_A(k,Z))$ and $\Ext^{l}_A(k,Z))$ are finitely generated bigraded $\mathcal{T}^*$-module, and hence $C_1,C_2,C_3$ and $C_4$ are so. Set
	 \[D_1= \text{Image}\big(\Ext^l_A(k,V(1,0))\rightarrow \Ext^l_A(k,Y)\big);\]
	\[D_2= \text{Image}\big(\Ext^l_A(k,Y)\rightarrow\Ext^l_A(k,V) \big).\]
	 Now let us break each of the exact sequence \ref{e14} and \ref{e15} in two parts. 
	\begin{equation}\label{e19}
	0\longrightarrow C_1 \longrightarrow \Ext^l_A(k,V(1,0))\longrightarrow D_1\longrightarrow 0.
	\end{equation}
	\begin{equation}\label{e20}
	0\longrightarrow D_1\longrightarrow \Ext^l_A(k,Y)\longrightarrow C_2\longrightarrow 0.
	\end{equation}
	\begin{equation}\label{e21}
	0\longrightarrow C_3 \longrightarrow \Ext^l_A(k,Y) \longrightarrow D_2\longrightarrow 0.
	\end{equation}
	\begin{equation}\label{e22}
0\longrightarrow D_2\longrightarrow \Ext^l_A(k,V)\longrightarrow C_4\longrightarrow 0.
	\end{equation}

 Let $\mathcal{T}_{++}$ be the ideal of $\mathcal{T}^*$ consisting of all sums of homogeneous elements $x_n\in \mathcal{T}^*_n$ such that $n_i\geq 1$, for $i=1,2$. Now considering the corresponding long eaxct sequence in local cohomology of ~\ref{e19},~\ref{e20},~\ref{e21},~\ref{e22} and by Theorem ~\ref{jk1}  there exist $(n_0,i_0) \in\mathbb{N}^2_0$ such that
\begin{equation}\label{e23}
H^0_{\mathcal{T}_{++}}(\Ext^l_A(k,V))_{n+1,i}\cong H^0_{\mathcal{T}_{++}}(\Ext^l_A(k,V))_{n,i}~\text{for all}~(n,i)\geq (n_0,i_0).
\end{equation}

 Set $L=H^0_{\mathcal{T}_{++}}(\Ext^l_A(k,V))$. Now using Lemma~\ref{ll2}~ on $L_{(n,i)}$ for fixed $n\geq0$ and then by ~\ref{e23} there exist $(n_\alpha,i_\alpha) \in\mathbb{N}^2_0$ such that
 	$$\Ass_A(L_{n,i})=\Ass_A(L_{n_\alpha,i_\alpha})~\text{for all}~ (n,i)\in \mathbb{N}^2_0~\text{with} ~(n,i)\geq (n_\alpha,i_\alpha).$$

  Since $\Ass_A\big(\Ext^l_A(k,V)\big)$ is either $\m$ or empty, by virtue of ~\ref{sirkatz} there exist $(n',i')\in \mathbb{N}^2_0$ such that 
   $$\Ass_A\big(\Ext^l_A(k,V)_{(n,i)}\big)=\Ass_A\big(\Ext^l_A(k,V)_{(n',i')}\big)~\text{for all}~ (n,i)\geq (n',i').$$
   
   Now the result follows from a well-known fact: for an $A$-module M, $\Ass_A(M)$ is non empty ~\ff~ $M\neq0$. Hence we are done .
 \end{proof}
 
We now give proof of our main result. We restate it for the convenience of the reader. 
 \begin{theorem}
 Let $(A,\m)$ be a local complete intersection ring. Let $M$ and $N$ be two finitely generated $A$-module, let $I$ be an ideal of $A$. Then for fixed $t=0,1$ we have  $$\depth\big(\Ext^{2i+t}_A(M,N/{I^nN})\big)~\text{is constant for all}~ n,i\gg0.$$	
 \end{theorem}
 \begin{proof}
 	It is well known that for a \fg $A$-module $E$, 
 	\begin{equation}\label{i13-2}
 	\depth_A{(E)}=\depth_{\hat{A}}(\hat{E}),
 	\end{equation}
 	 where $\hat{A}$ and $\hat{E}$ are completion of $A$ and $E$ respectively. Since $\hat{A}$ is flat over $A$ we have for each fixed $n,i\geq0$ and for each fixed $t=0,1$ $$\widehat{\Ext^{2i+t}_A(M,N/{I^nN})}=\Ext^{2i+t}_{\hat{A}}(\hat{M},\widehat{N/{I^nN}}).$$

 	  So we may assume $A$ is complete. So there exist a regular local ring $Q$ with $A=Q/{(f)}$ and $f=f_1,\ldots,f_{c}$ is a $Q$-regular sequence. Also we will have $\projdim_{Q}(M)<\infty$, since $M$ is a \fg $A$-module and so as $Q$-module.
 	We prove the Theorem for $t=0$ only. For $t=1$ the proof is similer.
 	
 	It is well known that for fixed $n,i$ 
 	\begin{equation*}
 	\depth\big({\Ext^{2i}_{A}(M,N/{I^nN})}\big) = \min\{ l : \Ext^l_{A}\big( k,\Ext^{2i}_{A}(M,N/{I^nN})\big)\neq 0 \}.
 	\end{equation*}

 	For each fixed $(n,i)\in \mathbb{N}^2_0$ and for each $l\geq 0$ set
 	 $$E^{l}_{(n,i)}:=\Ext^l_{A}\big( k,\Ext^{2i}_{A}(M,N/{I^nN}) \big).$$

 	 By Theorem~\ref{t3} for each fixed $l$ there exist some $(n_l,i_l)\in \mathbb{N}^2_0$ such that 
 	\[\text{either}~E^{l}_{(n,i)}=0 ~\text{for all}~ (n,i)\geq(n_l,i_l)~\text{or}~E^{l}_{(n,i)}\neq0 ~\text{for all}~ (n,i)\geq(n_l,i_l).\]

  Let $(\hat{n},\hat{i})$ be the maximimum of such $(n_j,i_j)$ for $0\leq j \leq \dim(A)$. Let $ \eta= \min\{l|E^{l}_{(\hat{n},\hat{i})}\neq0\}$. Hence we have $\depth\big({\Ext^{2i}_A(M,N/{I^nN})}\big)=\eta$ for all  $(n,i)\geq (\hat{n},\hat{i})$.
 \end{proof}
\section{Polynomial growth of Bass numbers}
\s \label{hhhh1}
Let $R$ be a Noetherian ring. Let $M$ be a finitely generated $R$-module. For any prime ideal $\p$ of $R$, the $j$-th Bass numbers of $M$ \wrt ~$\p$ is the $k(\p)$-dimension of $\Ext^j_{R_{\p}}\big(k(\p),M_{\p}\big)$, denoted by $\mu_j(\p,M)$, where $k(\p)$ is residue field of $A_{\p}$. Throughout this section we are assuming Set-up~\ref{1}. Let $M$  and $N$ be two finitely generated $A$-module with $\projdim_Q(M)<\infty$, let $I$ be an ideal of $A$. Set $\mathbf{T}$ as in  Note~\ref{no1}~ and  $\mathcal{T}$ as in  Note~\ref{no3}. Here we analyze the Bass number of $\Ext^{2i+t}_A(M,N/{I^nN})$ for fixed $t=0,1$. We have proved here for any prime $\p$ of $A$ and  for fixed $t=0,1$ the numbers $\mu_j\big(\p,\Ext^{2i+t}_A(M,N/{I^nN})\big)$ are given by polynomial in $(n,i)$ with rational coefficients for all sufficiently large $(n,i)$.
\begin{theorem}\label{t2}
Along with hypothesis as in~\ref{hhhh1} further assume $Q$ is a local ring with residue field $k$. Then for a fixed integer $l$ and for every fixed $t=0,1$ we have 
	\[ \lambda_A\bigg( \Ext^l_A\big(k,\Ext^{2i+t}_A(M,N/{I^nN})\big) \bigg)\]
	are given by polynomial in $(n,i)$ with rational coefficients for all sufficiently large $(n,i)$.
\end{theorem}
\begin{proof}
	For each $n\geq0$, consider the exact sequence
	\[0\rightarrow I^nN \rightarrow N\rightarrow N/I^nN\rightarrow 0,\] 
	which induces an exact sequence of $A$-modules(for each $i$):
	\[\Ext^i_A(M,I^nN)\longrightarrow \Ext^i_A(M,N)\longrightarrow \Ext^i_A(M,N/I^nN) \longrightarrow \Ext^{i+1}_A(M,I^nN).\]
	Taking direct sum over $n,i$ and using the naturality of the Eisenbud operators $t_j$, we have an exact sequence:  
	$$\bigoplus_{n,i\geq 0}\Ext^i_A(M,I^nN)\longrightarrow \bigoplus_{n,i\geq 0}\Ext^i_A(M,N)\longrightarrow \bigoplus_{n,i\geq 0}\Ext^i_A(M,N/I^nN)$$
	$$ \longrightarrow \bigoplus_{n,i\geq 0}\Ext^{i+1}_A(M,I^nN)$$
	 of bigraded $\mathcal{T}$-modules.
	
	\[ \text{Set}~U=\bigoplus_{n,i\geq 0}U_{n,i}=\bigoplus_{n,i\geq 0}\Ext^i_A(M,I^nN);~ B=\bigoplus_{n,i\geq 0}B_{n,i}=\bigoplus_{n,i\geq 0}\Ext^i_A(M,N);~\text{and}\]
	\[V=\bigoplus_{n,i\geq 0}V_{n,i}=\bigoplus_{n,i\geq 0}\Ext^i_A(M,N/I^nN).\]
	So we have  an exact sequence $U\longrightarrow B\longrightarrow V\longrightarrow U(0,1).$ Now setting $X=$ Image$(U\longrightarrow B)$, $Y=$ Image$(B\longrightarrow V)$ and $Z=$ Image$(V\longrightarrow U(0,1))$, we obtain the short exact sequences:
	\begin{equation}\label{e1}
	0\longrightarrow X\longrightarrow B\longrightarrow Y\longrightarrow 0.
	\end{equation}
	\begin{equation}\label{e2}
	0\longrightarrow Y\longrightarrow V\longrightarrow Z\longrightarrow 0.
	\end{equation}
	Applying $\Hom_A(k,-)$ on \ref{e1} and \ref{e2}, we get the following exact sequences:
	\begin{equation}\label{e3}
	0\longrightarrow C_1 \longrightarrow \Ext^l_A(k,B)\longrightarrow \Ext^l_A(k,Y)\longrightarrow C_2\longrightarrow 0,
	\end{equation}
	\begin{equation}\label{e4}
	0\longrightarrow C_3 \longrightarrow \Ext^l_A(k,Y) \longrightarrow \Ext^l_A(k,V)\longrightarrow C_4\longrightarrow 0,
	\end{equation}
	where we set $C_1=$Image$(\Ext^{l}_A(k,X)\rightarrow \Ext^l_A(k,B))$, $C_2=$Image$(\Ext^{l}_A(k,Y)\rightarrow \Ext^{l+1}_A(k,X))$, $C_3=$Image$(\Ext^{l-1}_A(k,Z)\rightarrow \Ext^l_A(k,Y))$, $C_4=$Image$(\Ext^{l}_A(k,V)\rightarrow \Ext^{l}_A(k,Z))$.
	By virtue of \ref{t1}  ~$U$ is finitely generated bigraded $\mathcal{T}$-module, and hence $X$ and $Z$ are so. This implies that $\Ext^l_A(k,X))$, $\Ext^{l+1}_A(k,X))$, $\Ext^{l-1}_A(k,Z))$ and $\Ext^{l}_A(k,Z))$ are finitely generated bigraded $\mathcal{T}$-module, and hence $C_1,C_2,C_3$ and $C_4$ are so. Now $C_j$ is annhilated by $\m$ for $j=1,2,3,4$, where $\m$ is the unique maximal ideal of $A$ . So for $j=1,2,3,4$  each $(C_j)_{(n,i)}$ is finitely generated $k$-module for all $(n,i)\in \mathbb{N}^2$. Hence as an $A$-module $(C_j)_{(n,i)}$ has finite length for all $(n,i)\in \mathbb{N}^2$ and $j=1,2,3,4$. Therefore, by applying the Hilbert-Serre Theorem to the bigraded $\mathcal{T}$-modules $C_1,C_2,C_3$ and $C_4$, we obtain
	\begin{equation}\label{e5}
	\sum_{n,i\geq 0}\lambda_A(C_j)_{(n,i)}z^i\omega^n= \frac{P_{C_j}(z,\omega)}{(1-z^2)^c(1-\omega)^r}~~ \text{for}~~ j=1,2,3,4,
	\end{equation}
	for some polynomials $P_{C_j}(z,\omega)$ over $\Z$ for $ j=1,2,3,4$.
	
	Fix $n\geq 0$. Now by~\ref{gg}, $\tilde{B}=\bigoplus_{i\geq 0}B_{n,i}= \bigoplus_{i\geq 0}\Ext^i_A(M,N)$ is finitely generated graded $\mathbf{T}:=A[t_1,\ldots, t_c]$-module and hence $\Ext^l_A(k,\tilde{B})$ is also so. As $\Ext^l_A(k,\tilde{B})$ is annhilated by $\m$, $\Ext^l_A(k,B_{n,i})$ has finite length as $A$-module for each $i\geq0$. Therefore, again by the Hilbert-Serre Theorem, we have for each fixed $n\geq 0$:
	\[\sum_{i\geq0}\lambda_A\big(\Ext^l_A(k,B_{n,i})\big)z^i=\frac{\tilde{P}_{\tilde{B}}(z)}{(1-z^2)^c}\]
	for some polynomial $\tilde{P}_{\tilde{B}}(z)\in \Z[z]$. Multiplying both side of the above equation by $\omega^n$ and taking sum over $n\geq0$, we get 
	\begin{equation}\label{e6}
	\sum_{n,i\geq0}\lambda_A\big(\Ext^l_A(k,B_{n,i})\big)z^i\omega^n=\frac{P_B(z,w)}{(1-z^2)^c(1-w)^r},
	\end{equation}
	where $P_B(z,w)=\tilde{P}_{\tilde{B}}(z)(1-\omega)^{r-1} \in \Z[z,\omega]$.
	
	Now considering \ref{e3} we have an exact sequence of $A$-modules:
	\begin{equation}\label{e7}
	0\longrightarrow (C_1)_{(n,i)} \longrightarrow \Ext^l_A(k,B_{(n,i)})\longrightarrow \Ext^l_A(k,Y_{(n,i)})\longrightarrow (C_2)_{(n,i)}\longrightarrow 0,
	\end{equation}
	for each $n,i\geq0$. So the additivity of the length function gives
	\[\lambda_A \big((C_1)_{(n,i)}\big)-\lambda_A\big(\Ext^l_A(k,B_{(n,i)})\big)+ \lambda_A\big(\Ext^l_A(k,Y_{(n,i)})\big)-\lambda_A\big( (C_2)_{(n,i)}\big)=0,\]
	for each $n,i\geq0$. Multiplying both side by $z^i\omega^n$, then taking sum over $n,i\geq0$, and using \ref{e5} and \ref{e6} we obtain 
	\begin{equation}\label{e9}
	\sum_{n,i\geq0}\lambda_A\big(\Ext^l_A(k,Y_{(n,i)})\big)z^i\omega^n = \frac{P_Y(z,w)}{(1-z^2)^c(1-w)^r}
	\end{equation}
	where $P_Y(z,w)=P_{C_2}(z,\omega)-P_{C_1}(z,\omega)+P_B(z,w) \in \Z[z,\omega].$
	
	Now considering \ref{e4} we have an exact sequence of $A$-modules:
	\begin{equation}\label{e10}
	0\longrightarrow (C_3)_{(n,i)} \longrightarrow \Ext^l_A(k,Y_{(n,i)}) \longrightarrow \Ext^l_A(k,V_{(n,i)})\longrightarrow (C_4)_{(n,i)}\longrightarrow 0
	\end{equation}
	for each $n,i\geq0$. So the additivity of the length function gives
	\[\lambda_A \big((C_3)_{(n,i)}\big)-\lambda_A\big(\Ext^l_A(k,Y_{(n,i)})\big)+ \lambda_A\big(\Ext^l_A(k,V_{(n,i)})\big)-\lambda_A\big((C_4)_{(n,i)}\big)=0,\]
	for each $n,i\geq0$. Multipling both side by $z^i\omega^n$, then taking sum over $n,i\geq0$, and using \ref{e5} and \ref{e9} we obtain 
	\begin{equation}\label{e11}
	\sum_{n,i\geq0}\lambda_A\big(\Ext^l_A(k,V_{(n,i)})\big)z^i\omega^n = \frac{P_V(z,w)}{(1-z^2)^c(1-w)^r},
	\end{equation}
	where $P_V(z,w)=P_{C_4}(z,\omega)-P_{C_3}(z,\omega)+P_Y(z,w) \in \Z[z,\omega].$
	Therefore it follows that
	 \[ \lambda_A\bigg(\Ext^l_A\big(k,\Ext^{2i}_A(M,N/{I^nN})\big)\bigg)~ \text{and}~\lambda_A\bigg(\Ext^l_A\big(k,\Ext^{2i+1}_A(M,N/{I^nN})\big)\bigg)\]
	are given by polynomial in $n,i$ with rational coefficients for all sufficiently large $(n,i)$.
\end{proof}

\begin{remark}
	By Theorem~\ref{t2}~ we have for a fixed integer $l$ and for every fixed $t=0,1$  
	$$ \lambda_A\bigg( \Ext^l_A\big(k,\Ext^{2i+t}_A(M,N/{I^nN})\big) \bigg)=f(n,i)~ \text{for all}~ (n,i)\gg0
	,$$ 
	where $f(n,i)\in \mathbb{Q}[n,i]$.
	
	 As $\bigoplus_{n,i\geq 0}\Ext^l_A\big(k,\Ext^{2i+t}_A(M,N/{I^nN})\big)$ is not ~\fg~ over any Noetherian ring $R$, the assertion $\lambda_A\bigg( \Ext^l_A\big(k,\Ext^{2i+t}_A(M,N/{I^nN})\big) \bigg)=f(n,i)$ for all $(n,i)\gg 0$ does not imply that either $f(n,i)>0$ for $(n,i)\gg0$ or $f(n,i)=0$ for $(n,i)\gg0$.
\end{remark}

We now give
\begin{proof}[Proof of Theorem \ref{growth-bass}]
	We note that $A_\p$ is a local complete intersection ring. By taking a completion of $A_\p$ \wrt \ $\p A_\p$ the result follows from Theorem~\ref{t2}. 
\end{proof}

\section*{Acknowledgements}
The first author would like to thank UGC, MHRD, Govt.\,of India. for providing financial support for this study.


\begin{thebibliography}{AAAA}

\bibitem[B79-1]{B79-1}
M.~Brodmann,
\emph{Asymptotic stability of $Ass(M/I^nM)$},
Proc. Amer. Math. Soc. 74 (1979), no. 1, 16–-18.

\bibitem[B79-2]{B79-2}
M.~Brodmann,
\emph{The asymptotic nature of the analytic spread},
Math. Proc. Cambridge Philos. Soc. 86 (1979), no. 1, 35-–39.


\bibitem[Eis80]{Eis80} D. Eisenbud, {\it Homological algebra on a complete intersection, with an application to group representations}, Trans. Amer. Math. Soc. {\bf 260} (1980), 35--64.

 \bibitem[GP]{GP}
D.~Ghosh and T.~J.~Puthenpurakal, 
\emph{Asymptotic prime divisors over complete intersection rings},
Math. Proc. Cambridge Philos. Soc. 160 (2016), no. 3, 423–-436. 
 Corrigendum: Math. Proc. Cambridge Philos. Soc. 163 (2017), no. 2, 381–-384
 .
\bibitem[Gul74]{G} T.~H.~Gulliksen, {\it  A change of ring theorem with applications to Poincar\'{e} series and intersection multiplicity}, Math. Scand. {\bf 34} (1974), 167--183.

\bibitem[Jkv02]{Jkv02} A.~V. Jayanthan and J~.K. Verma, {\it Grothendieck-Serre formula and bigraded
	Cohen–Macaulay Rees algebras}, J. Algebra {\bf 254} (2002), 1--20.
	
	
\bibitem[Tk]{Tk} Daniel Katz and T.~J.~Puthenpurakal, {\it Quasi-finite modules and asymptotic prime divisors}, J. Algebra {\bf 280} (2013), 18--29.

	
	
\bibitem[Put13]{Put13} T.~J.~Puthenpurakal, {\it On the finite generation of a family of Ext modules}, Pacific. J. Math {\bf 266} (2013), 367--389.


\bibitem[Se17]{Se17}
T.~Se, 
\emph{Covariant functors and asymptotic stability},
J. Algebra 484 (2017), 247–-264.


\bibitem[West]{West} E.~West, {\it Primes associated to multigraded modules}, J. Algebra {\bf 271} (2004), 427--453.


\end{thebibliography}
\end{document}